\documentclass[12pt,a4paper,reqno]{amsart}
\usepackage{amssymb}
\usepackage{amscd}
\usepackage{enumerate}
\usepackage{graphicx}
\usepackage{siunitx}
\usepackage{tikz-cd}
\usepackage{bm}
\numberwithin{equation}{section}

\usepackage{mathtools}
\usepackage[tableposition=top]{caption}
\usepackage{booktabs,dcolumn}

\DeclareFontFamily{OT1}{rsfs}{}
\DeclareFontShape{OT1}{rsfs}{n}{it}{<-> rsfs10}{}
\DeclareMathAlphabet{\mathscr}{OT1}{rsfs}{n}{it}

\theoremstyle{plain}

\newtheorem{theorem}{Theorem}[section]
\newtheorem{proposition}[theorem]{Proposition}
\newtheorem{lemma}[theorem]{Lemma}
\newtheorem{corollary}[theorem]{Corollary}
\newtheorem{conjecture}[theorem]{Conjecture}

\theoremstyle{definition}

\newtheorem{definition}[theorem]{Definition}
\newtheorem{remark}[theorem]{Remark}

\newcommand\R{\mathbb{R}}
\newcommand\Z{\mathbb{Z}}
\newcommand\N{\mathbb{N}}

\begin{document}

\title{Infinite partial sumsets in the primes}

\author{Terence Tao}
\address{UCLA Department of Mathematics, Los Angeles, CA 90095-1555.}
\email{tao@math.ucla.edu}

\author{Tamar Ziegler}
\address{Einstein Institute of Mathematics, The Hebrew University of Jerusalem,
Jerusalem, 91904, Israel}
\email{tamarz@math.huji.ac.il}



\begin{abstract}  We show that there exist infinite sets $A = \{a_1,a_2,\dots\}$ and $B = \{b_1,b_2,\dots\}$ of natural numbers such that $a_i+b_j$ is prime whenever $1 \leq i < j$.
\end{abstract}

\maketitle


\section{Introduction}

In \cite{erdos}, Erd\H{o}s asked the question of whether, given a subset $A$ of the natural numbers $\N = \{1,2,3,\dots\}$ of positive upper density, there existed an infinite subset $B$ of $A$ and a natural number $t$ such that $b+b'+t \in A$ for all distinct $b,b'$ in $B$.  This conjecture was recently proven in \cite{kmrr}, using techniques from ergodic theory and topological dynamics.  It is natural to ask whether the same result holds if the set $A$ is replaced by the primes ${\mathcal P} = \{2,3,5,\dots\}$.  Our first result is the observation that one can answer this question affirmatively assuming the Dickson--Hardy--Littlewood conjecture, which we now pause to recall.

\begin{definition}[Admissible and prime-producing tuples]  A tuple $(h_1,\dots,h_k)$ of natural numbers is said to be an \emph{admissible tuple} if it avoids at least one residue class mod $p$ for each prime $p$.  A tuple $(h_1,\dots,h_k)$ is said to be \emph{prime-producing} if there are infinitely many $n$ such that $n+h_1,\dots,n+h_k$ are simultaneously prime.
\end{definition}

It is easy to see that every prime-producing tuple must be admissible.  In the converse direction, we have

\begin{conjecture}[Dickson--Hardy--Littlewood conjecture]\label{qhl}  Every admissible tuple \\$(h_1,\dots,h_k)$ is prime-producing.
\end{conjecture}

This conjecture is a special case of a conjecture of Dickson \cite{dickson}, with the case $k=2$ being an older conjecture of de Polignac \cite{dep}; the specific case of the tuple $(0,2)$ is of course the infamous twin prime conjecture.  In \cite{hl}, Hardy and Littlewood proposed a more quantitative asymptotic for $\sum_{n \leq x} \Lambda(n+h_1) \dots \Lambda(n+h_k)$ which implies Conjecture \ref{qhl}, but we will not need this stronger conjecture here.  Conjecture \ref{qhl} is immediate from the infinitude of primes when $k=1$, but the existence of even a single prime-producing pair $(h_1,h_2)$ was only established relatively recently by Zhang \cite{zhang}.  Shortly afterwards, Maynard \cite{maynard} showed the existence of prime-producing $k$-tuples for any $k$; indeed, every admissible $k$-tuple contains a prime-producing $l$-tuple for some $l \gg \log k$. In \cite{polymath} it was shown that for $k=50$ one can take $l=2$.

Assuming Conjecture \ref{qhl} one can show an analogue of the Erd\H{o}s question for the primes:

\begin{theorem}[Infinite restricted sumsets in the shifted primes]\label{main1}  Assume Conjecture \ref{qhl}.  Then there exists an infinite set $B$ of primes such that $b+b'+1$ is prime for every distinct $b,b' \in B$.
\end{theorem}

Theorem \ref{main1}  turns out to be easily established by iteratively applying Conjecture \ref{qhl} to increasingly large admissible tuples; we give the short proof in Section \ref{easy}.  We remark that it was previously shown (among other\footnote{For instance, Granville also shows assuming Conjecture \ref{qhl} that the primes contain $k$-fold sumsets $A_1+\dots+A_k$ of infinite sets $A_1,\dots,A_k$ for any $k$.} things) in \cite{Granville} that Conjecture \ref{qhl} implied that the primes contained the sumset $A+B$ of two infinite sumsets $A, B$; Theorem \ref{main1} provides a new proof of this claim.  Also, it follows from the results in \cite{balog} (see also \cite[Examples 4, 9]{green-tao}; the companion results in \cite{gt-mobius}, \cite{gtz} are not needed here in this ``complexity one'' situation) that the above theorem holds unconditionally if ``infinite set'' is replaced by ``arbitrarily large finite sets''.

\begin{remark} In view of the results in \cite{kmrr}, one might also ask if Theorem \ref{main1} continued to hold if one replaced the set of primes ${\mathcal P}$ by some positive (relative) density subset.  However, this is not the case.  Indeed, if $A$ was any subset of the primes for which there existed an infinite set $B$ for which $b+b'+1 \in A$ for all distinct $b,b' \in B$, then the set $A$ would necessarily contain bounded gaps infinitely often (since if $b_1,b_2$ are two elements of $B$ then $A$ must contain infinitely many pairs of the form $n+b_1, n+b_2$).  However, it is easy to construct a subset $A$ of the primes of relative density $1$ whose gaps between consecutive entries goes to infinity; for instance, if $h \colon \R^+ \to \R^+$ is any function with $\frac{h(x)}{\log x} \to 0$ and $h(x) \to \infty$ as $x \to \infty$, one can take $A$ to consist of all primes $p \geq 100$ such that there are no primes in the interval $[p+1, p + h(p)]$, since it follows from \cite[Theorem 2]{gallagher} that $A$ has relative density $1$ amongst all the primes, yet the gaps between consecutive entries of $A$ clearly go to infinity. A similar remark applies to Theorem \ref{main2} (or Corollary \ref{half}) below.
\end{remark}

Now we turn to what one can say unconditionally, i.e., without assuming Conjecture \ref{qhl}.  Our main result in this direction is

\begin{theorem}[Ascending chain of prime-producing tuples]\label{main2}  There exists an infinite sequence $h_1 < h_2 < \dots$ of natural numbers, such that the $k$-tuple $(h_1,\dots,h_k)$ is prime-producing for every $k$.
\end{theorem}

One can rephrase this theorem in an equivalent form:

\begin{corollary}[Primes contain half of an infinite sumset]\label{half}  There exist infinite sequences $a_1 < a_2 < \dots$ and $b_1 < b_2 < \dots$ of natural numbers such that $a_i + b_j$ is prime whenever $1 \leq i < j$.
\end{corollary}

\begin{proof}  Take $a_i=h_i$ to be the sequence from Theorem \ref{main2}.  By that theorem, for each $j$ there exist infinitely many $b_j$ such that $a_i + b_j$ is prime for all $1 \leq i < j$.  In particular, one can choose the $b_j$ to be increasing in $j$, and the claim follows.
\end{proof}

In the converse direction, it is immediate that Corollary \ref{half} implies Theorem \ref{main2}, since the $k$-tuple $(a_1,\dots,a_k)$ is clearly prime-producing for every $k$.  The abstract equivalence of these results (with the primes replaced by more general set) was previously observed in \cite{chou}.  In the language of \cite{chou}, these results assert that the primes are not an $R_W$-set, while in the language of \cite{ruppert}, they assert that the primes are not a translation-finite set\footnote{This answers a question of Yemon Choi in {\tt mathoverflow.net/questions/3347} in the negative.}. Inserting Theorem \ref{main2} into the results of \cite{ruppert}, we also conclude that there exist bounded functions supported on the primes that are not weakly almost periodic in $\ell^\infty(\Z)$.	 See \cite{choi} for further discussion of the properties of translation-finite sets.

Corollary \ref{half} also implies the result of Maynard \cite{maynard} that arbitrarily long prime-producing tuples exist.  Indeed, Theorem \ref{main2} will be established by adapting the sieve of Maynard \cite{maynard} (as modified in \cite{polymath}, \cite{bfm}), as well as using the intersectivity lemma of Bergelson \cite{berg}; we do so in Sections \ref{intersective-sec}, \ref{sieve-sec}.

The proof of Corollary \ref{half} in fact allows one to place the sequence $a_i$ inside any specified infinite admissible set; for instance one could require the $a_i$ to be odd square numbers if desired.  See Proposition \ref{inf-ad}.  As a consequence of this stronger statement, one can establish that the orbit closure of the primes is uncountable; see Corollary \ref{uncount}.

A similar result to Corollary \ref{half}, in which the $a_i+b_j$ are required to be sums of two squares rather than prime was recently\footnote{We thank James Maynard for this reference.} established by McGrath \cite[Theorem 1.4]{mcgrath}, with applications to the failure of quantum ergodicity for high dimensional flat tori \cite{jakobson}.  As with our result, one can place the $a_i$ inside any specified admissible set, such as (any constant multiple of) the set of odd squares, though for technical reasons the result in \cite{mcgrath} requires that the $a_i$ are divisible by $4$.  The methods of proof are similar (in particular relying on the Maynard sieve, the second moment method, and the pigeonhole principle), and so it seems likely that the arguments here could be adapted to give a slightly different\footnote{One technical difference is that the arguments in \cite{mcgrath} require \emph{asymptotics} for two-point correlations associated to a weight function adapted to sums of two squares, whereas for our approach \emph{upper bounds} on these two-point correlations suffice, mainly thanks to our use of the Bergelson intersectivity lemma which is not used in \cite{mcgrath}.} proof of the results in \cite{mcgrath} (using the half-dimensional version of the Maynard sieve developed in that paper).

\subsection{Notation}

We use $X = O(Y)$, $X \ll Y$, or $Y \gg X$ to denote the bound $|X| \leq CY$ for an absolute constant $C$, and if we say that $X = o(Y)$ as $N \to \infty$, we mean that $|X| \leq c(N) Y$ for some quantity $c(N)$ depending on $N$ (and possibly some other fixed parameters) that goes to zero as $N \to \infty$ (holding all other parameters fixed).

For any natural numbers $b,W$, we use $b\ (W)$ to denote the residue class $b+W\Z$ in the cyclic group $\Z/W\Z$.

\subsection{Acknowledgements}

The first author is supported by NSF grant DMS-1764034 and by a Simons Investigator Award.  The second author is supported by a grant from the Institute for Advanced Study and by ISF grant 2112/20.  Part of this research was conducted at the Institute for Advanced Study.  We thank Joel Moreira for discussions leading to Section \ref{remarks-sec}, Vitaly Bergelson for informing us of the reference \cite{dkkl}, James Maynard for informing us of the reference \cite{mcgrath}, Yemon Choi for supplying references on translation-finite sets, Mariusz Lem\'anczyk for pointing out a gap in a previous version of the proof of Corollary \ref{uncount}, Joel David Hamkins for providing an argument allowing us to strengthen that corollary, and Yemon Choi and Keiju Sono for corrections.

\section{Proof of Theorem \ref{main1}}\label{easy}

For any $k \geq 1$, define a \emph{good $k$-tuple} to be an increasing $k$-tuple $(p_1,\dots,p_k)$ of primes $3 < p_1 < \dots < p_k$ of primes larger than three, such that $p_i + p_j + 1$ is prime, $p_{i+1}-p_i >2$ for all $i\ge 1$, and $p_i$ does not divide $p_j+2$ for all $1 \leq i < j \leq k$.  For instance, any prime $p>3$ forms a good $1$-tuple $(p)$.  The key proposition to iterate is

\begin{proposition}[Extending a good tuple]  Assume Conjecture \ref{qhl}, and let $k \geq 1$.  Then any good $k$-tuple $(p_1,\dots,p_k)$ can be extended to a good $k+1$-tuple $(p_1,\dots,p_{k+1})$.
\end{proposition}

\begin{proof}  We first verify that the $k+2$-tuple $(0, 2, p_1+1,\dots,p_k+1)$ is admissible.  If $p$ is a prime not equal to any of the $p_1,\dots,p_k$, then this $k+2$-tuple avoids the residue class $1 \hbox{ mod } p$.  If instead $p = p_i$ for some $1 \leq i \leq k$, then the $k+2$-tuple avoids the residue class $-1 \hbox{ mod } p_i$; this is clear for $0,2,p_1+1,\dots,p_i+1$ since $p_i > 3$, and also for $p_j+1$, $j>i$ since we are assuming that $p_i$ does not divide $p_j+2$.  This establishes admissibility. Applying Conjecture \ref{qhl}, we can find $n > p_k+2$ such that $n, n+2, n+p_1+1,\dots,n+p_k+1$ are all prime.  Setting $p_{k+1} \coloneqq n$, we conclude in particular that $p_1 < \dots < p_{k+1}$ is a good tuple, as claimed.
\end{proof}

Iterating this proposition starting from (say) $p_1=5$, we can find an infinite sequence $p_1 < p_2 < \dots$ of primes such that $p_i + p_j + 1$ is prime for all $1 \leq i < j$, and Theorem \ref{main1} follows.

\begin{remark}\label{dynam}  One can view the above construction using the dynamical systems framework of \cite{kmrr} (and indeed our arguments were initially inspired by this framework).  Let $X$ denote\footnote{This space is denoted $X_{1_{\mathcal P}}$ in \cite{dkkl}.} the orbit closure of the primes ${\mathcal P}$, that is to say the closure in the product topology of the shifts ${\mathcal P}+t, t \in \Z$ of the primes, viewed as elements of the Cantor space $2^\Z = \{ A: A \subset \Z\}$.  This is a compact space endowed with a shift homeomorphism $T \colon X \to X$ defined by $TA \coloneqq A-1$.  Using the quantitative form of the Hardy--Littlewood conjecture and standard upper bound sieves, $X$ can in fact be described explicitly as the collection of all shifts ${\mathcal P}+t$ of ${\mathcal P}$, together with the collection of all admissible tuples of integers (either finite or infinite).  Following \cite[Definition 2.1]{kmrr}, define an \emph{Erd\H{o}s progression} to be a triple $(x_0,x_1,x_2)$ of points $x_0,x_1,x_2 \in X$, such that there exists an increasing sequence $n_1 < n_2 < \dots$ of integers such that $T^{n_i} x_0 \to x_1$ and $T^{n_i} x_1 \to x_2$ as $i \to \infty$.  By repeating the arguments at the end of \cite[\S 2]{kmrr}, to establish the conclusion of Theorem \ref{main1}, it suffices to find an Erd\H{o}s progression $({\mathcal P}, x_1, x_2)$ with $x_1$ in the open cylinder set
$$ \{ A \subset \Z: 0 \in A \}$$
and $x_2$ in the open cylinder set
$$ \{ A \subset \Z: 1 \in A \}$$
(that is to say, $0 \in x_1$ and $1 \in x_2$).  One can check that the construction given above indeed generates an Erd\H{o}s progression $({\mathcal P}, x_1, x_2)$ with
$$ \{0, p_1+1, p_2+1, \dots\} \subset x_1$$
and
$$ \{1\} \subset x_2$$
(indeed, under a quantitative form of the Hardy--Littlewood conjecture and standard upper bound sieves, one can turn these inclusions into equalities).
\end{remark}

\section{An application of Bergelson's intersectivity lemma}\label{intersective-sec}

We now begin the proof of Theorem \ref{main2}.  We will need the following application of the Maynard sieve, which we will prove in later sections.

\begin{proposition}[Application of Maynard sieve]\label{mayprop}  Let $J_1,\dots,J_I \geq 2$ be natural numbers,
and let $(h_{i,j})_{1 \leq i \leq I; 1 \leq j \leq J_i}$ be an admissible $\sum_{i=1}^I J_i$-tuple.  Let $\theta_1,\dots,\theta_I > 0$ be real numbers with $\sum_{i=1}^I \theta_i \leq 1$.  Then if $N$ is sufficiently large depending on all previous parameters, there exists a probability measure $\nu$ on the integers in $[N,2N]$ such that
$$ \nu(\{ n:  n + h_{i,j} \in {\mathcal P} \}) \gg \theta_i \frac{\log J_i}{J_i}$$
for all $1 \leq i \leq I$ and $1 \leq j \leq J_I$, and
$$ \nu( \{ n: n + h_{i,j}, n + h_{i,j'} \in {\mathcal P} \} ) \ll \left(\theta_i \frac{\log J_i}{J_i}\right)^2$$
for all $1 \leq i \leq I$ and $1 \leq j < j' \leq J_I$.
\end{proposition}

Let us assume this proposition for now and complete the proof of Theorem \ref{main2}. Set $J_i \coloneqq 2^{2^i}$ for all $i \geq 1$, let $\Sigma \subset \N^2$ denote the countably infinite index set
$$ \Sigma \coloneqq \{ (i,j) \in \N^2: 1 \leq j \leq J_i \}$$
and let $(h_{i,j})_{(i,j) \in \Sigma}$ be an infinite sequence of distinct odd squares, which we can assume to be increasing in the sense that $h_{i,j} < h_{i',j'}$ whenever $i < i'$.   The point of using odd squares is that any finite tuple of the $(h_{i,j})_{(i,j) \in \Sigma}$ is necessarily admissible, since for every odd prime $p$ there is at least one quadratic nonresidue mod $p$.
By the preceding proposition with $\theta_i \coloneqq 2^{-i}$, we can find a sequence $N_I$ going to infinity as $I \to \infty$, and a probability measure $\nu_I$ on the integers in $[N_I, 2N_I]$ for every $I$, such that for each $I$ we have
\begin{equation}\label{o1}
 \nu_I(\{ n:  n + h_{i,j} \in {\mathcal P} \}) \gg \frac{1}{J_i}
\end{equation}
for all $1 \leq i \leq I$ and $1 \leq j \leq J_i$, and
\begin{equation}\label{o2}
\nu_I( \{ n: n + h_{i,j}, n + h_{i,j'} \in {\mathcal P} \} ) \ll \frac{1}{J^2_i}
\end{equation}
for all $1 \leq i \leq I$ and $1 \leq j < j' \leq J_i$.

Next, we apply a variant of the Furstenberg correspondence principle \cite{furst} to pass to a suitable limit as $I \to \infty$. Consider the map $\phi \colon \N \to 2^{\Sigma}$ defined by
$$ \phi(n) \coloneqq \{ (i,j) \in \Sigma: n+h_{i,j} \in {\mathcal P} \},$$
and let $\mu_I \coloneqq \phi_* \nu_I$ be the (Borel) probability measure on the Cantor space $2^\Sigma$ formed by pushing forward $\nu_I$ by $\phi$; thus
$$ \mu_I(E) \coloneqq \nu_I \{ n: \phi(n) \in E \}$$
for any measurable $E \subset 2^{\Sigma}$.  In particular, if we let $E_{i,j} \subset 2^{\Sigma}$ denote the (clopen) events
$$ E_{i,j} \coloneqq \{ A \in 2^{\Sigma}: (i,j) \in A \} $$
then from \eqref{o1}, \eqref{o2} we have
\begin{equation}\label{p1}
 \mu_I( E_{i,j} ) \gg \frac{1}{J_i}
\end{equation}
for all $1 \leq i \leq I$ and $1 \leq j \leq J_i$, and
\begin{equation}\label{p2}
\mu_I( E_{i,j} \cap E_{i,j'} ) \ll \frac{1}{J^2_i}
\end{equation}
for all $1 \leq i \leq I$ and $1 \leq j < j' \leq J_i$.

By the Banach-Alaoglu theorem, there exists a (Borel) probability measure $\mu$ on $2^\Sigma$ that is a limit point (in the vague topology) of the $\mu_I$. In particular, from \eqref{p1}, \eqref{p2} we have
\begin{equation}\label{q1}
 \mu( E_{i,j}) \gg \frac{1}{J_i}
\end{equation}
for all $i \geq 1$ and $1 \leq j \leq J_i$, and
\begin{equation}\label{q2}
 \mu( E_{i,j} \cap E_{i,j'}) \ll \frac{1}{J^2_i}
\end{equation}
for all $i \geq 1$ and $1 \leq j < j' \leq J_i$.

Now we use the second moment method (as in \cite{bfm}).  For each $i \geq 1$, we have from \eqref{q1}, \eqref{q2} that
$$ \int_{2^\Sigma} \sum_{j=1}^{J_i} 1_{E_{i,j}}\ d\mu \gg \frac{J_i}{J_i} = 1$$
and
\begin{align*}
 \int_{2^\Sigma} \left(\sum_{j=1}^{J_i} 1_{E_{i,j}} \right)^2\ d\mu &\ll \frac{J_i^2}{J_i^2} + \int_{2^\Sigma} \sum_{j=1}^{J_i} 1_{E_{i,j}}\ d\mu \\
&\ll \int_{2^\Sigma} \sum_{j=1}^{J_i} 1_{E_{i,j}}\ d\mu.
\end{align*}
Since $\sum_{j=1}^{J_i} 1_{E_{i,j}}$ is supported on the event
$$ E_i \coloneqq \bigcup_{j=1}^{J_i} E_{i,j}$$
we have from Cauchy--Schwarz that
$$ \left(\int_{2^\Sigma} \sum_{j=1}^{J_i} 1_{E_{i,j}}\ d\mu\right)^2 \leq \mu(E_i) \int_{2^\Sigma} \left(\sum_{j=1}^{J_i} 1_{E_{i,j}} \right)^2\ d\mu.$$
Putting these inequalities together, we conclude that
$$ \mu(E_i) \gg 1$$
for all $i$.

Now we invoke the following result of Bergelson \cite{berg}:

\begin{lemma}[Bergelson intersectivity lemma]\label{bcl-1}\cite[Theorem 1.1]{berg}  Let $E_1,E_2,E_3,\dots$ be events in a probability space $(X,\mu)$ such that
$$
 \inf_i \mu(E_i) > 0.
$$
Then there exists $1 \leq i_1 < i_2 < \dots$ such that
$$ \mu(E_{i_1} \cap \dots \cap E_{i_k} ) > 0$$
for all $k$.  In fact one can take the set $\{i_1,i_2,\dots\}$ to have positive upper density.
\end{lemma}

Applying this lemma, we can find $1 \leq i_1 < i_2 < \dots$ such that
$$ \mu(E_{i_1} \cap \dots \cap E_{i_k} ) > 0$$
for all $k$.  By repeated application of the pigeonhole principle, we may thus find $1 \leq j_r \leq J_{i_r}$ for all $r \geq 1$ such that
$$ \mu(E_{i_1,j_1} \cap \dots \cap E_{i_k,j_k} ) > 0$$
for all $k$.  For each such $k$, we conclude from vague convergence that
$$ \mu_I(E_{i_1,j_1} \cap \dots \cap E_{i_k,j_k} ) > 0$$
for infinitely many $I$.  This implies that for infinitely many $I$, there exists $n \in [N_I, 2N_I]$ such that
$$ n+h_{i_1,j_1}, \dots, n + h_{i_k,j_k} \in {\mathcal P}.$$
Since the $N_I$ go to infinity as $I \to \infty$, we conclude that the tuple $(h_{i_1,j_1},\dots,h_{i_k,j_k})$ is prime-producing for every $k$.  This establishes Theorem \ref{main2} assuming Proposition \ref{mayprop}.

It remains to establish Proposition \ref{mayprop}.  This will be done in the next section.

\begin{remark}  We did not utilize the conclusion that $\{i_1,i_2,\dots\}$ had positive upper density, or equivalently that $i_k = O(k)$ for infinitely many $k$.  If one does so, and also optimizes the values of $\theta_i, J_i$, one can ensure\footnote{We thank Freddie Manners and Vitaly Bergelson for suggesting this refinement.} that for any admissible $h_1 < h_2 < \dots$, there exists a subsequence $h_{i_1} < h_{i_2} < \dots$ with $i_k \leq \exp(k^{1+o(1)})$ for infinitely many $k$, such that $h_{i_1},\dots,h_{i_k}$ is prime-producing for every $k$.  Setting the $h_i$ to be the odd squares (for instance), we can thus ensure that the sequence $a_i$ in Theorem \ref{half} obeys the bound $a_i \leq \exp(i^{1+o(1)})$ for infinitely many $i$; we leave the details to the interested reader.  Any significant quantitative improvement to the Maynard sieve would lead to improved bounds on the $i_k$ or the $a_i$.
\end{remark}

\section{The Maynard sieve}\label{sieve-sec}

We first present a version of the Maynard sieve \cite{maynard} which is a slight variant of the version presented in \cite{bfm}.
Let $J_1,\dots,J_I, \theta_i, h_{i,j}, N$ be as in Proposition \ref{mayprop}.  We define $\Sigma_I \subset \N^2$ to be the finite set
$$ \Sigma_I \coloneqq \{ (i,j) \in \N^2: i \leq I; j \leq J_i \}.$$
Let $F \colon [0,+\infty)^\Sigma_I \to \R$ be a smooth compactly supported function be chosen later, that is assumed to be not identically zero and of the form
\begin{equation}\label{ftt}
 F( (t_{(i,j)})_{(i,j) \in \Sigma_I}) = \sum_{\alpha \in A} \prod_{(i,j) \in \Sigma_I} F_{\alpha,i,j}(t_{i,j})
\end{equation}
for some finite index set $A$ and some smooth compactly supported functions $F_{\alpha,i,j}: [0,+\infty) \to \R$ such that
\begin{equation}\label{ftt-supp}
 \sum_{(i,j) \in \Sigma_I} \sup \{t: F_{\alpha,i,j}(t) \neq 0\} \leq \frac{1}{10}.
\end{equation}
We define the sieve weights
\begin{equation}\label{lad}
 \lambda_{(d_{i,j})_{(i,j) \in \Sigma_I}} \coloneqq \prod_{(i,j) \in \Sigma_I} \mu(d_{i,j}) \sum_{\alpha \in A} \prod_{(i,j) \in \Sigma_I} F'_{\alpha,i,j}\left( \frac{\log d_{i,j}}{\log N}\right)
\end{equation}
and then define the integral quantity
$$ I(F) \coloneqq \int_0^\infty \dots \int_0^\infty F( (t_{i,j})_{(i,j) \in \Sigma_I} )\ \prod_{(i,j) \in \Sigma_I} dt_{i,j}.$$
Also, for any $(i_1,j_1) \in \Sigma_I$, we define the integral quantity
$$ J_{(i_1,j_1)}(F) \coloneqq \int_0^\infty \dots \int_0^\infty \left(\int_0^\infty  F( (t_{i,j})_{(i,j) \in \Sigma_I} )\ dt_{i_1,j_1}\right)^2\ \prod_{(i,j) \in \Sigma_I \backslash \{(i_1,j_1)\}} dt_{i,j},$$
and similarly for any distinct $(i_1,j_1), (i_2,j_2) \in \Sigma_I$, we define the integral quantity
$$ L_{(i_1,j_1), (i_2,j_2)}(F) \coloneqq \int_0^\infty \dots \int_0^\infty \left(\int_0^\infty \int_0^\infty  F( (t_{i,j})_{(i,j) \in \Sigma_I} )\ dt_{i_1,j_1}dt_{i_2,j_2} \right)^2\ \prod_{(i,j) \in \Sigma_I \backslash \{(i_1,j_1), (i_2,j_2)\}} dt_{i,j}.$$
We set\footnote{One can in fact work with significantly larger values of $W$ if desired (as large as a small power of $N$), as long as one excludes primes associated with a potential exceptional modulus: see \cite{bfm}.}
$$ W \coloneqq \prod_{p \leq \log\log\log N} p.$$
Since $(h_{i,j})_{(i,j) \in \Sigma_I}$ is admissible, we see from the Chinese remainder theorem that we may find a residue class $b\ (W)$ (depending on $N$, of course) such that $b+h_{i,j}\ (W)$ is a primitive residue class for all $(i,j) \in \Sigma_I$.  Fix this choice of $b$.  We let
$$ k \coloneqq \sum_{i=1}^I J_i$$
denote the cardinality of $\Sigma_I$, and introduce the quantity
$$ B \coloneqq \frac{\phi(W)}{W} \log N.$$

The following proposition is a slight variant of \cite[Lemma 4.5]{bfm}, and is proven in essentially the same fashion.

\begin{proposition}[Maynard sieve]  If $w \colon \N \to \R^+$ denotes the sieve weight
$$ w(n) \coloneqq 1_{N < n \leq 2N: n = b \ (W)} \left(\sum_{(d_{i,j})_{(i,j) \in \Sigma_I}: d_{i,j}|n+h_{i,j} \forall (i,j) \in \Sigma_I} \lambda_{(d_{i,j})_{(i,j) \in \Sigma_I}}\right)^2$$
then one has the estimates
\begin{align}
\sum_n w(n) &= (1+o(1)) \frac{N}{W} B^{-k} I(F) \label{may-1}\\
\sum_n 1_{\mathcal P}(n+h_{(i_1,j_1)}) w(n) &= (1+o(1)) \frac{N}{W} B^{-k} J_{(i_1,j_1)}(F) \label{may-2}\\
\sum_n 1_{\mathcal P}(n+h_{(i_1,j_1)}) 1_{\mathcal P}(n+h_{(i_2,j_2)}) w(n) &\ll \frac{N}{W} B^{-k} L_{(i_1,j_1), (i_2,j_2)}(F) \label{may-3}
\end{align}
for all distinct $(i_1,j_1), (i_2,j_2) \in \Sigma_I$ in the limit as $N \to \infty$ (keeping all other parameters fixed).
\end{proposition}

\begin{proof}(Sketch)  The first asymptotic \eqref{may-1} follows from expanding out the weight $w(n)$ and applying \cite[Theorem 3.6(i)]{polymath} much as in the calculations after \cite[(85)]{polymath}.  The second asymptotic \eqref{may-2} similarly follows from \cite[Theorem 3.5(i)]{polymath} (with $\theta=1/2$, as per the Bomberi--Vinogradov inequality), again following the calculations after \cite[(85)]{polymath}.  To obtain the final upper bound \eqref{may-3}, we use the argument from the proof of \cite[Lemma 4.5(iii)]{bfm}, namely we note the pointwise bound
$$ 1_{\mathcal P}(n+h_{(i_2,j_2)}) w(n) \leq \tilde w(n)$$
where $\tilde w$ is defined as in $w$, except that the functions $F_{\alpha,i_2,j_2}$ used (via \eqref{lad}) to define $w$ are replaced with the function $\tilde F_{\alpha,i_2,j_2}$ defined by
$$\tilde F_{\alpha,i_2,j_2}(t) \coloneqq F_{\alpha,i_2,j_2}(0) G(t)$$
where $G \colon [0,+\infty) \to \R$ is a fixed smooth function supported on $[0,1/10]$ which equals $1$ at the origin (the exact choice of $G$ is not important as we are not trying to optimize the implied constant in \eqref{may-3}).  Inserting this bound, we can bound the left-hand side of \eqref{may-3} by the left-hand side of \eqref{may-2} (with $w$ replaced by $\tilde w$), and then by repeating the calculations used to prove \eqref{may-2} (as in the proof of \cite[Lemma 4.5(iii)]{bfm}) we obtain the desired bound.
\end{proof}

As a corollary of this proposition, if we take $\nu$ to be the probability measure on $[N,2N]$ with density
$$ \frac{w(n)}{\sum_n w(n)},$$
then for $N$ sufficiently large, we have
\begin{equation}\label{sam-1}
 \nu(\{ n:  n + h_{i_1,j_1} \in {\mathcal P} \}) \gg \frac{J_{(i_1,j_1)}(F) }{I(F)}
\end{equation}
and
\begin{equation}\label{sam-2}
 \nu( \{ n: n + h_{i,j}, n + h_{i,j'} \in {\mathcal P} \} ) \ll \frac{L_{(i_1,j_1), (i_2,j_2)}(F)}{I(F)}
\end{equation}
for all distinct $(i_1,j_1), (i_2,j_2) \in \Sigma_I$.

For any fixed $i$, we can use an argument from \cite{bfm} (which in turn is based on calculations in \cite{maynard}).
To use these bounds, we use a further lemma from \cite{bfm} to construct a suitable preliminary weight function $F_i$ for each index $i$.

\begin{lemma}[Maynard sieve weight]  For each $i=1,\dots,I$, there exists a smooth compactly supported function $F_i: [0,+\infty)^{J_i} \to \R$, not identically zero, of the form
$$ F_i( t_{i,1},\dots,t_{i,J_i}) = \sum_{\alpha \in A_i} \prod_{j=1}^{J_i} \tilde F_{\alpha,i,j}(t_{i,j})$$
for some smooth compactly supported $\tilde F_{\alpha,i,j} \colon [0,+\infty) \to \R^+$ and some finite index set $A_i$, such that
$$ \sum_{j=1}^{J_i} \sup \{t: F_{\alpha,i,j}(t) \neq 0\} \leq \frac{1}{10}$$
for all $\alpha \in A_i$, and such that
$$ J_{(i,j)}(F_i) \gg \frac{\log J_i}{J_i} I(F_i)$$
and
$$ L_{(i,j), (i,j')}(F_i) \ll \left(\frac{\log J_i}{J_i}\right)^2 I(F_i)$$
for all distinct $j,j' \in \{1,\dots,J_i\}$, where the quantities $I(F_i), J_{(i,j)}(F_i), L_{(i,j), (i,j')}(F_i)$ are defined similarly to $I(F), J_{(i_1,j_1)}(F), L_{(i_1,j_1), (i_2,j_2)}(F)$ but with the index set $\Sigma_I$ replaced by the slice $\{ (i,j): 1 \leq j \leq J_i\}$.
\end{lemma}

\begin{proof}  One can assume $J_i$ to be large, as the claim is trivial for bounded $J_i \geq 2$.  The claim now follows from \cite[Lemma 4.6]{bfm} (with $k=J_i$ and $\rho = \delta = 1/10$, say) and some obvious relabeling.
\end{proof}

If one now defines
$$ F( (t_{i,j})_{(i,j) \in \Sigma_I} ) \coloneqq \prod_{i=1}^I F_i( t_{i,1}/\theta_i,\dots, t_{i,J_i}/\theta_i )$$
then one easily verifies that $F$ is of the required form \eqref{ftt} (with the condition \eqref{ftt-supp} being obeyed), and from Fubini's theorem and a change of variables one has
$$ J_{(i,j)}(F) \gg \theta_i \frac{\log J_i}{J_i} I(F)$$
and
$$ L_{(i,j), (i,j')}(F) \ll \left(\theta_i \frac{\log J_i}{J_i}\right)^2 I(F)$$
for all distinct $(i,j), (i,j') \in \Sigma_I$.  Inserting these bounds into \eqref{sam-1}, \eqref{sam-2}, we obtain Proposition \ref{mayprop} and hence Theorem \ref{main2}.

\section{Consequences for the orbit closure of the primes}\label{remarks-sec}

We thank Joel Moreira for suggesting the following remarks.

The orbit closure $X$ of the primes in Remark \ref{dynam} is clearly contained in the union of the set $T^\Z {\mathcal P} = \{ {\mathcal P} + t: t \in \Z \}$ of finite translates of the primes, and the collection ${\mathcal A}$ of all (finite and infinite) admissible subsets of $\Z$, these sets are disjoint since the primes are not admissible (they do not avoid any residue class mod $2$).  Observe that a finite tuple is prime-producing if and only if it is contained in an element of $X \cap {\mathcal A}$. As mentioned in the previous remark, the quantitative form of the Hardy--Littlewood prime tuples conjecture would imply (using standard upper bound sieves) that in fact $X = T^\Z {\mathcal P} \cup {\mathcal A}$; see also \cite[Remark 1.2]{dkkl} for an alternate proof.  Of course, this statement is out of reach unconditionally; even showing that $\{0,2\}$ for instance was in $X$ would imply the twin prime conjecture (and is morally equivalent to it).  However, our methods of proof do show

\begin{proposition}\label{inf-ad}  Let $A$ be an infinite admissible set of integers.  Then there exists an infinite subset $A'$ of $A$ that is contained in $X$.  In particular, $X$ contains at least one infinite admissible set.
\end{proposition}

This can be compared with the result of Maynard \cite{maynard} that any finite admissible $k$-tuple contains a prime-producing subtuple of cardinality $\gg \log k$.

\begin{proof}  (Sketch) Repeat the proof of Theorem \ref{main2}, but with all the $h_{i,j}$ chosen from $A$ (rather than from the odd squares).  The Maynard sieve calculations used to prove Proposition \ref{mayprop} also yield the additional bound
$$ \nu(\{ n:  n + h \in {\mathcal P} \}) = o(1)$$
as $N \to \infty$ for any fixed $h$ equal to any of the $h_{i,j}$; in fact the right-hand side can be replaced with $O(B^{-1})$ (where the implied constants can depend on the parameters $I$, $J_i$, $h_{i,j}$, $\theta_i$).  The proof of Theorem \ref{main2} then produces an infinite sequence $h_{i_1,j_1},h_{i_2,j_2},\dots$ in $A$ with the property that for any $k$, there exist infinitely many $n$ such that
$$ n + h_{i_1,j_1}, \dots, n+h_{i_k,j_k} \in {\mathcal P}$$
but also that $n+h \not \in {\mathcal P}$ for any $h$ in $\{-k,\dots,k\}$ not equal to any of the $h_{i,j}$.  Taking weak limits of ${\mathcal P}-n$, we conclude that the orbit closure of $X$ contains a subset of $A$ that contains all of the $h_{i_1,j_1}, h_{i_2,j_2}, \dots$, and is hence infinite, giving the claim.
\end{proof}

This proposition has the following corollary, answering a question of Marius Lema\'ncyzk (private communication):

\begin{corollary}\label{uncount}  Every infinite admissible subset $A$ of integers contains a family of infinite subsets in $X$ of the cardinality of the continuum.  In particular, $X$ is uncountable.
\end{corollary}

We remark that the uncountability of ${\mathcal A}$ was previously observed in \cite[Remark 2.41]{dkkl}.

\begin{proof}  We use an argument of Joel David Hamkins\footnote{{\tt mathoverflow.net/a/440671/766}}.  Use the elements of the countably infinite set $A$ to enumerate the entries of an infinite binary tree $B$.  The number of branches of this tree has the cardinality of the continuum.  Each branch determines an infinite subset $A'$ of $A$, which inherits the admissibility of $A$, thus by Proposition \ref{inf-ad} contains a further infinite set $A''$ in $X$.  Since any two branches of $B$ have only finitely many nodes in common, the sets $A''$ are all distinct, and the claim follows.
\end{proof}

\begin{remark}  An alternate way to prove Corollary \ref{uncount} (communicated to us by Mariusz Lema\'nczyk and Forte Shinko) proceeds by  first using Proposition \ref{inf-ad} and a variant of the Cantor diagonal argument to show that $A$ contains uncountably many subsets in $X$, and then appealing to the Cantor--Bendixon theorem \cite[Theorem 6.5]{kechris} to show that this family of subsets (being a closed uncountable subset of a Cantor space) contains a non-empty perfect set and hence has the cardinality of the continuum.
\end{remark}

\end{document}